\documentclass{amsart}

\usepackage{hyperref}

\newtheorem{theorem}{Theorem}
\newtheorem{lemma}[theorem]{Lemma}
\newtheorem{corollary}[theorem]{Corollary}

\theoremstyle{definition}
\newtheorem{definition}[theorem]{Definition}
\newtheorem{example}[theorem]{Example}

\newtheorem{question}[theorem]{Question}

\theoremstyle{remark}
\newtheorem{remark}[theorem]{Remark}

\newcommand{\SL}{\mathrm{SL}}
\newcommand{\PSL}{\mathrm{PSL}}
\newcommand{\Z}{\mathbb{Z}}
\newcommand{\F}{\mathbb{F}}
\newcommand{\FF}{\mathcal{F}}

\begin{document}

\title{Bounding the residual finiteness of free groups}

\author{Martin Kassabov}
\address{Cornell University, Department of Mathematics,
590 Malott Hall, Ithaca, NY 14853, USA}
\curraddr{School of Mathematics, University of Southampton,
University Road, Southampton, SO17 1BJ, UK}
\email{kassabov@math.cornell.edu}
\thanks{The first author was partially funded by  grants
from National Science Foundation DMS~0600244, 0635607 and~0900932.}

\author{Francesco Matucci}
\address{University of Virginia, Department of Mathematics,
Charlottesville, VA 22904, USA}
\email{fm6w@virginia.edu}



\begin{abstract}
We find a lower bound to the size of finite groups detecting
a given word in the free group, more precisely
we construct a word $w_n$ of length $n$ in non-abelian free groups 
with the property that $w_n$ is the identity on all
finite quotients of size $\sim n^{2/3}$ or less.
This improves on a previous result of Bou-Rabee and McReynolds
quantifying the lower bound of the residual finiteness of free groups.
\end{abstract}

\maketitle

A group $G$ is called \emph{residually finite} if for any $w\in G$, $w\ne 1$
there exists a finite group $H$ and
a homomorphism $\varphi:G \to H$ such that $\varphi(w)\ne 1$. We will say that
such group $H$ \emph{detects} the element $w$.
One way to quantify this property is look at the minimal size of a finite
group $H$ which detects a given word, and study the behavior of this function.
We define the following natural growth function to measure
the residual finiteness of a group
(introduced by Bou-Rabee in~\cite{bourabee1}):
\[
k_G(w):=\min \{|H|  \mid \mbox{there exists } \pi:G \to H, \pi(w) \ne 1 \}
\]
and
\[
F_G^S(n):=\max \{ k_G(w) \mid |w|_S \le n \},
\]
where $S$ is a generating set of the group $G$ and $|w|_S$ denotes
the word length of $w$ with respect to the generating set $S$.

We also write $f_1 \preceq f_2$ to mean that there exists a $C$ such
that $f_1(n) \le C f_2(Cn)$ for all $n$,
and we write $f_1 \simeq f_2$ to mean $f_1 \preceq f_2$ and
$f_2 \preceq f_1$. It is easy to see
that if $G$ is finitely generated the
growth type of the function $F_G^S$ does not depend on the set $S$,
assuming that it is finite.

In this short note, we will focus on the free group $\FF_k$ on $k$ generators.
Bou-Rabee~\cite{bourabee1} and Rivin~\cite{rivin1}
have shown that $F_{\FF_k}(n) \preceq n^3$.
Both proofs are obtained by embedding the
free group $\FF_k$ into $\SL_2(\Z)$
and then finding a suitable prime $p$ such that a given word does not vanish
in the quotient $\SL_2(\Z/p\Z)$, for a slightly
different proof see Remark~\ref{thm:cubic-upper}.

Recently, Bou-Rabee and McReynolds~\cite{boumcrey1}
have shown (see Corollary~\ref{thm:third})
that
$F_{\FF_k}(n) \succeq n^{1/3}$.
We improve %
this lower bound, establishing the following result:

\begin{theorem}
$F_{\FF_k}(n) \succeq n^{2/3}$. \label{thm:main}
\end{theorem}

The main new ingredient in the proof of Theorem~\ref{thm:main}
is a result of Lucchini (Theorem~\ref{thm:lucchiin}) about finite permutation groups.

\subsection*{Known results on the growth function $F_G$}

The following lemma of Bou-Rabee~\cite{bourabee1}
explains why the function $F_G^S$ is independent
on the generating set:
\begin{lemma}
\label{thm:remove-decoration}
Let $G$ be residually finite group generated by a finite set $S$.
If $H$ is a subgroup of $G$ generated by a finite set $T$, then
$H$ is residually finite and $F_H^T(n) \preceq F_G^S(n)$.
\end{lemma}
The proof of this lemma is based on the observation that any nontrivial
element in $H$ is also a nontrivial in $G$ therefore it can be detected by
some finite quotient.

Applying Lemma~\ref{thm:remove-decoration} twice with $G=H$,
we see that we can drop the decoration $S$ in $F_G^S(n)$.
Moreover, this implies that the growth functions for
all non-abelian finitely generated free groups
are all equivalent.
Hence, we will restrict to study the function $F_{\FF_2}(n)$
for the free group on two generators.

\begin{example}
Most of the words in the free group $\FF_k$
can be detected using relatively small quotients. For example,
the word
$
w:=z^2 y^{23} x^{36} y^{33} z^{-26}
$
in $\FF_3$
has length $120$, but can be detected through the homomorphism
$\varphi:\FF_3 \to \Z/3\Z$
which maps $\varphi(y)=1 \pmod{3}$ and $\varphi(x)=\varphi(z)=0 \pmod{3}$.
\end{example}

This argument works, because the
word $w$ has nontrivial image in the abelianization $\Z^3$ of $\FF_3$.
In general, if a word $w \in \FF_k$ is detected by the abelianization $\Z^k$,
it is also detected in a suitable (relatively small) finite quotient of $\Z^k$.
Bou-Rabee~\cite{bourabee1} used the prime number theorem to show:

\begin{theorem}
\label{thm:Zk}
$F_{\Z^k}(n) \simeq \log(n)$.
\end{theorem}

Therefore, for any $w\in \FF_k \setminus [\FF_k,\FF_k]$ one has
$k_{\FF_k}(w) \leq C \log |w|$. This is sufficient to show that the
value of the function $k_{\FF_k}$ on a random word of length $n$ is
bounded above by a constant~\cite{rivin1}.

Theorem~\ref{thm:Zk} can be generalized to nilpotent and soluble groups, but
one needs to replace the logarithmic bound of Theorem \ref{thm:Zk} with a polylogarithmic one.
This suggests that the words $w$ in $\FF_k$, where $k_{\FF_k}(w)$ is large,
are the ones which lie deep in the lower central/derived  series.

\subsection*{Rephrasing in terms of laws}

One possible way to study the function $F_G$ is to notice that
an element $w$ can be detected by a finite quotient of size at most $n$
if and only if
$
w\not\in G_{n}
$,
where $G_n$ is the intersection of all normal subgroups of $G$
of index at most $n$. Thus, one can derive properties of the function
$F_G$ by estimating the word length of shortest element in the group $G_n$.

For example, let $G$ be the free group $\FF_k$.
It is well known~\cite{lubotzkysegalbook}
that $\FF_k$ has 
$a^{\lhd}_n(\FF_k) < n^{4k\log n}$ normal
subgroups of index at most $n$ (for all large enough $n$),
therefore the intersection
\[
\FF_{k,n} = \bigcap_{H \unlhd \FF_k, [\FF_k:H] \leq n} H
\]
has index at most
\[
[\FF_k : \FF_{k,n}] \leq n^{a^{\lhd}_n} < n^{n^{4k\log n}},
\]

for every $n$ large enough. This shows that $\FF_{k,n}$ has an element $w_n$
of length at most
$\log [\FF_k : \FF_{k,n}] \le  e^{4k\log ^2 n} \log n$.
By construction $w_n$ can not be detected by
any finite group of size at most $n$, therefore $k_{\FF_k}(w_n) > n$.
This shows that

\begin{lemma}
\label{thm:somelower}
$F_{\FF_k}(n) \succeq e^{\sqrt {\log n}}.$
\end{lemma}

Of course this lower bound is not optimal, there are two
reasons for that -- first the bound $[\FF_k : \FF_{k,n}]$ is very far from the
correct one; and second the shortest element in $\FF_{k,n}$ is very likely to have
length significantly smaller than $\log [\FF_k : \FF_{k,n}]$.

\medskip

An equivalent, but slightly more convenient way to study the group $\FF_{k,n}$ is to
consider laws in finite groups.

\begin{definition}
Given a group $\Gamma$ an \emph{identity} or \emph{law} in
$\Gamma$ on $k$ letters
is a word $w(x_1,\ldots,x_k)$ in the free group $\FF_k$ such that
$w(g_1,\ldots,g_k)=1$, for all elements $g_1,\ldots,g_k \in \Gamma$.
We denote by $L_{\Gamma,k} \lhd \FF_k$ the subgroup of all identities in
$\Gamma$ and by $\alpha_k(\Gamma)$ the length of the shortest
identity in a group $\Gamma$.
\end{definition}

It is easy to see that
\[
\FF_{k,n} = \bigcap L_{\Gamma,k}
\]
where the intersection is taken over all isomorphic classes of
finite groups $\Gamma$ of size at most $n$. In particular one has
\begin{itemize}
\item if $\alpha_k(\Gamma) > l$ for some finite group $\Gamma$ of order $n$, then
$\FF_{k,n}$ does not contain any words of length $l$, which implies that
$F_{\FF_k}(l) \leq n$;
\item if a word $w\in \FF_k$ is an identity in any finite group
of order at most $n$ then $k_{\FF_k}(w)  \geq n$, i.e.,
$F_{\FF_k}(|w|) \geq n$.
\end{itemize}

\bigskip

These two observations allow us to obtain upper and lower bounds for $F_{\FF_k}$
by using results about identities in finite groups.
The following result of  Hadad~\cite{hadad1} can be used to
obtain an upper bound for $F_{\FF_k}(n)$.
\begin{theorem}
\label{thm:UziSL}
The length of the shortest identities in
$\SL_{2}(\F_q)$  and $\PSL_{2}(\F_q)$ satisfies
\[
\frac{q-1}{3} \le \alpha_k(\PSL_{2}(\F_q)) \le
             \alpha_k(\SL_{2}(\F_q))  < 10 (q + 2).
\]
\end{theorem}
\begin{remark}
\label{thm:cubic-upper}
Let $w$ be a word in the free group of length $n$. By the above result
$w$ is not an identity in the group $\PSL_2(\F_p)$ for any prime $p > 3n +1$,
therefore $w$ is not in the kernel of some map $\FF_k \to \PSL_2(\F_p)$, i.e.,
$k_{\FF_2}(w) \leq |\PSL_2(\F_p)| = (p^3-p)/2$. Thus, we have
$F_{\FF_k}\left(\frac{p-1}{3}\right)  \leq (p^3-p)/2$, i.e.,
$F_{\FF_k}(n)\preceq n^3$.
\end{remark}

\begin{remark}
It seems the methods~\cite{hadad1} can be used to show that
the shortest identity satisfied in all groups of the form $\SL_2(R)$,
where $R$ is a finite commutative ring of size at most $N$,
has length $C N^{2}$. If this is indeed the case then
one can improve the upper bound $F_{\FF_k}(n) \preceq n^3$
to $F_{\FF_k}(n) \preceq n^{3/2}$.
\end{remark}

On the other side, it is easy see that $w=x^{n!}$
is an identity in any group of order at most $n$.
This can be used to obtain a lower bound for $F_{\FF_k}(n)$,
which is weaker then Lemma~\ref{thm:somelower}.
However, there is an easy way to improve this bound by constructing
a shorter identity by using the following lemma~\cite{hadad1},
which plays a central role in the proof of the
upper bound in Theorem~\ref{thm:UziSL}.

\begin{lemma}
\label{thm:commutator}
Let $r_1 > \ldots > r_m$ be a finite sequence of integers.
There exist a nontrivial word $w = w_{r_1,\dots,r_m}\in \FF_2$ of length at most
$4m^2\cdot (r_1+1)$ with the following property:
$w$ is a law in any finite group $\Gamma$ such that every
$\gamma \in \Gamma$ is a solution to at least one of these equations
\[
X^{r_1}=1, \ldots, X^{r_m}=1.
\]
\end{lemma}
\begin{proof}[Sketch of the proof:]
The nontrivial word $w = w_{r_1,\dots,r_m}$ 
is built
by taking a suitable iterated commutator of (conjugates of) the powers
$x^{r_1}, \ldots, x^{r_m}$ of the first letter $x$.
For example, if $m=4$ we build
\[
[x^{r_1}, (x^{r_2})^{y}], \qquad [x^{r_3},(x^{r_4	})^{y}],
\]
and then take their commutator
\[
w:=[[x^{r_1}, (x^{r_2})^{y}],[x^{r_3},(x^{r_4})^{y}]].
\]
For general $m$ one needs to be careful to create a
commutator word $w$ that has no internal cancellation.
\end{proof}

This lemma gives the following lower bound for the function $F_{\FF_k}(n)$.%
\footnote{This was obtained independently
by Bou-Rabee and McReynolds in~\cite{boumcrey1}.}

\begin{corollary}
\label{thm:third}
$F_{\FF_k}(n) \succeq n^{1/3}$.
\end{corollary}
\begin{proof}
The order of any element in a finite group
is at most the size of the group. By Lemma~\ref{thm:commutator},
we can build a word $v_n:=w_{n,n-1,\dots,1}$ which is an identity in any finite group of size $\leq n$.
Moreover the length of $v_n$ is
$|v_n| \leq 4n^2(n+1)$, thus
$F_{\FF_k}(4n^3 + 4n^2) > n$.
\end{proof}

\medskip

\subsection*{Improvement of Corollary~\ref{thm:third}}
This construction can be significantly improved if one combines it with
a result of Lucchini about permutation groups~\cite{lucchini1}.

\begin{theorem}\label{thm:lucchiin}
Let $\Gamma$ be a transitive permutation group of degree $n > 1$
whose point-stabilizer subgroup is cyclic. Then $|\Gamma| \le n^2 - n$.
\end{theorem}

\begin{remark}
Unlike many similar results about permutation groups,
the proof of Lucchini's result is elementary and does not rely on the Classification
of Finite Simple Groups.
\end{remark}

In the proof of Theorem~\ref{thm:main}  we will use the following
corollary to Lucchini's result%
\footnote{This was obtained independently
by Herzog and Kaplan in~\cite{herzogkaplan} in 2001 and became known to the authors
only after the completion of this work.},
which will give us information in case a group has an element of ``large order''.

\begin{corollary}
Let $\Gamma$ be a finite group and let $x \in \Gamma$ such that
$|x| \ge \sqrt{|\Gamma|}$. Then there is an integer $\ell < \sqrt{|\Gamma|}$
such that $\langle x^\ell \rangle \unlhd \Gamma$.
\label{thm:normal}
\end{corollary}

\begin{proof}
The action of $\Gamma$ on the cosets $\Gamma/\langle x\rangle$ is transitive. If $N$
is the kernel of the action,
then $N \le \langle x \rangle$ is cyclic and, by Lucchini's Theorem
$[\Gamma:N]<|\Gamma:\langle x \rangle|^2$. Since
\[
[\Gamma:\langle x \rangle][\langle x \rangle:N]=[\Gamma:N]<[\Gamma:\langle x \rangle]^2,
\]
then 
\[
\ell:=[\langle x \rangle:N]<[\Gamma:\langle x \rangle]\le \sqrt{|\Gamma|}.
\]

Thus, $\langle x^\ell \rangle \le N \unlhd \Gamma$,
since $N$ is cyclic, we have $\langle x^\ell \rangle \unlhd \Gamma$.
\end{proof}
\begin{remark}
The previous result shows that if a group $\Gamma$
has elements of order larger than $\sqrt{\Gamma}$,
there are restrictions on the structure of $\Gamma$. The
next natural step would be to study a group $\Gamma$ with
elements of order larger than $\sqrt[3]{\Gamma}$. We observe that the Classification
of Finite Simple Groups implies that any non-abelian finite simple group $S$
does not have elements of order more than
$|S|^{1/3}$, thus it seems likely that existence an
element in a finite group $\Gamma$ of order more $|\Gamma|^{1/3}$
implies restrictions on the structure of $\Gamma$.
\end{remark}

\begin{lemma}
\label{thm:lawN2}
Let $n$ be a positive integer, then the word $v_n \in \FF_2$ constructed in
Corollary~\ref{thm:third} is a law in every group of order $|\Gamma| \le \frac{1}{9}n^2$.
\end{lemma}

\begin{proof}
We want to show that the commutator word $v_n \in \FF_2$
is a law on any group
$\Gamma$ of order $|\Gamma| \le \frac{1}{9} n^2$, so we
evaluate $v_n(\gamma_1, \gamma_2)$ on any two elements
$\gamma_1, \gamma_2 \in \Gamma$. There are two cases:
\begin{itemize}
\item[(1)]
If $|\gamma_1| \le n$, then $\gamma_1^k =1$ for some $k<n$, thus
$v_n(\gamma_1, \gamma_2)=1$.

\item[(2)] If $|\gamma_1| > n$ then, by Corollary~\ref{thm:normal},
there is a power $\ell < n/3$ such that the cyclic group
$N:=\langle \gamma_1^\ell\rangle$ is normal in  $\Gamma$.
There exist at least two powers $\gamma_1^s,\gamma_1^t$,
for suitable $s < n/2< t  < n$ such that $\gamma_1^s,\gamma_1^t \in N$.
However, by construction $v_n$ is a commutator of two words $w'$ and $w''$
which are built as commutators conjugates of powers of $x^i$.
Since $N$ is normal and the powers $x^s$ and $x^t$
are involved in $w'$ and $w''$ respectively and we have
that $w'(\gamma_1,\gamma_2),w''(\gamma_1,\gamma_2) \in N$.
The group $N$ is abelian which implies that
\[
v_n(\gamma_1,\gamma_2)= [w'(\gamma_1,\gamma_2),w'(\gamma_1,\gamma_2)] \in [N,N] =1.
\]
\end{itemize}
In  both cases we have seen that $v_n(\gamma_1,\gamma_2)=1$, hence
$v_n$ is an identity in $\Gamma$.
\end{proof}

The previous Lemma immediately
implies that there exists a word of length $n$ in $\FF_2$
which cannot be detected by any group of size $n^{2/3}$ or less.
\begin{proof}[Proof of Theorem~\ref{thm:main}]
By Lemma~\ref{thm:lawN2} we have a word $v_n$ which is the identity in any finite group
of size $\le \frac{1}{9}n^2$ and, by Lemma~\ref{thm:commutator}, the length of $v_n$ is
$|v_n| \leq 4n^2(n+1)$. Therefore $F_{\FF_k}(4n^3 + 4n^2) > \frac{1}{9}n^2$.
\end{proof}

\begin{remark}
Theorem~\ref{thm:main} and Remark~\ref{thm:cubic-upper}
are also valid if the free group is replaced by a surface group.
The lower bound follows from Lemma~\ref{thm:remove-decoration},
Theorem~\ref{thm:main} and the observation that
any surface group contains a free subgroup.
Rivin has showed in~\cite{rivin1}
that the upper bound of Remark~\ref{thm:cubic-upper} can be extended
to surface groups.
\end{remark}

\subsection*{Open questions\label{sec:open-questions}}
From Theorem~\ref{thm:main} and Bou-Rabee's result on the upper bound,
one can ask the following
natural question:

\begin{question}
Is it true that $F_{\FF_k}(n) \simeq n$?
\end{question}

\begin{question}
What is the asymptotic behavior of the index of $[\FF_2:\FF_{2,n}]$
as a function of $n$?

Let $\FF_{2,n}^{\max}$ denote the intersection of all
maximal normal subgroup in $\FF_2$
of index at most $n$. Using estimates on the number of maximal normal subgroups, it can be shown that $[F_2:F_{2,n}^{\max}] \sim e^{n^{4/3}}$. Given that 
$[F_2:F_{2,n}] \ge [F_2:F_{2,n}^{\max}]$, this yields 
that $[F_2:F_{2,n}] \succeq e^{n^{4/3}}$. This may suggest, based on the weaker evidence of Lemma \ref{thm:somelower}, that $[F_2:F_{2,n}] \sim e^{n^\alpha}$, for some $\alpha$.
\end{question}

We define the \emph{divisibility function} of the free group
$D_{\FF_k}:\FF_k \setminus \{1\} \to \mathbb{N}$. The function is defined by
\[
D_{\FF_k}(w)=\min_{H \le \FF_k} \{[\FF_k:H] \mid w \not \in H\}.
\]
Bogopolski asked the following question in
the Kourovka notebook~\cite{kourovkanotebook}.

\begin{question}
Does there exist a $C=C(k)>0$ such that $D_{\FF_k}(w) \le C \log(|w|)$?
\end{question}

Bou-Rabee and McReynolds~\cite{boumcrey1}
showed that this question has a negative answer
establishing the following result on the lower bound.

\begin{theorem}
$\max_{|w| \le n}D_{\FF_k}(w) \not \preceq \log(n)$.
\end{theorem}

On the other hand, Buskin~\cite{buskin1}
has an upper estimate using Stallings automata:

\begin{theorem}
$D_{\FF_k}(w) \le  \frac{|w|}{2} + 2$.
\end{theorem}

\subsection*{Acknowledgments}
The authors thank Nikolai Nikolov,
Khalid Bou-Rabee, Ben McReynolds and Enric Ventura
for helpful discussions.
The authors would also like to thank
Oleg Bogopolski and Mikhail Ershov for providing the reference~\cite{buskin1}
and Avinoam Mann for providing the reference ~\cite{herzogkaplan}. Finally
the authors would like to thank the referees for a careful reading which improved the form
of this paper.

The authors gratefully acknowledge the support from
the Centre de Recerca Matem\`atica, Barcelona
and the Universitat Polit\'ecnica de Catalunya
where most of work has been done in June 2009.

\bibliographystyle{amsplain}

\end{document}